\documentclass[a4paper,11pt]{amsart}
\usepackage{graphicx}
\usepackage{amsmath}
\usepackage{amssymb}
\usepackage{oldgerm}
\usepackage[all]{xy}

\newcommand{\End}{\operatorname{End}\nolimits}

\newcommand{\Hom}{\operatorname{Hom}\nolimits}

\newcommand{\id}{\operatorname{id}\nolimits}

\newcommand{\soc}{\operatorname{soc}\nolimits}
\newcommand{\rad}{\operatorname{rad}\nolimits}
\newcommand{\cone}{\operatorname{cone}\nolimits}
\newcommand{\N}{\operatorname{\mathbb{N}}\nolimits}
\newcommand{\Z}{\operatorname{\mathbb{Z}}\nolimits}
\newtheorem{theo}{Theorem}[section]
\newtheorem{cor}[theo]{Corollary}
\newtheorem{lemma}[theo]{Lemma}

\newtheorem{defi}[theo]{Definition}
\newtheorem{rem}[theo]{Remark}

\begin{document}
\title{Finite and bounded Auslander-Reiten components in the derived Category}
\author{Sarah Scherotzke}
\address{Sarah Scherotzke \newline Mathematical Institute \\ University of Oxford \\ 24-29 St.\
Giles \\ Oxford OX1 3LB \\ United Kingdom}
\email{scherotz@maths.ox.ac.uk}
\date{\today}

 \maketitle

\begin{abstract}
We analyze Auslander-Reiten
components for the bounded derived category of a
finite-dimensional algebra. We classify derived categories
whose Auslander-Reiten quiver has either a finite stable component or a stable component
with finite Dynkin tree class or a bounded stable component. Their
Auslander-Reiten quivers are determined.  We also determine components
that contain shift periodic complexes.
\end{abstract}

\section{Introduction}
In this paper we analyze the Auslander-Reiten triangles in the
bounded derived category of a finite-dimensional indecomposable algebra $A$,
denoted by $D^b(A)$. The bounded derived category of a
finite-dimensional algebra is a triangulated category and
Auslander-Reiten triangles are triangles with analogous properties
to Auslander-Reiten sequences for finite-dimensional algebras.
The conditions for the existence of such triangles in $D^b(A)$ have
been determined by Happel in \cite{H2}.

Analogously to the classical Auslander-Reiten theory, which
applies to Artin algebras, we can define Auslander-Reiten
components of the bounded derived category. These are locally
finite graphs, where the vertices correspond to indecomposable
complexes in $D^b(A)$. We want to know how and if certain results
on finite and dimension-bounded Auslander-Reiten components of finite-dimensional algebras
extend to the bounded derived category.

\smallskip

In the second section we give a brief introduction
to derived categories and we introduce Auslander-Reiten triangles as
defined by Happel in \cite{H1}. In the third section we deduce some properties of Auslander-Reiten triangles
that will be used in the other sections.

In the fourth section we classify the bounded derived categories
that have either a stable finite Auslander-Reiten component, a
stable Auslander-Reiten component with finite Dynkin tree class or a
stable bounded Auslander-Reiten component. In all these cases the
Auslander-Reiten quiver is described completely. Also components with shift periodic modules are determined.

\smallskip

In the classical Auslander-Reiten theory finite components occur
if and only if the algebra has finite representation type. We show that finite stable
Auslander-Reiten components occur for $D^b(A)$ if and only if $A$
is simple. In this case each component of the
Auslander-Reiten quiver is isomorphic to $A_1$.

\smallskip

In the classical Auslander-Reiten theory we call an Auslander-Reiten component bounded if the dimension of the modules appearing in this component is bounded. Motivated by this definition, we introduce bounded Auslander-Reiten components for the bounded derived category and show that the following are equivalent

1) The bounded derived category of $A$ has finite representation type, i.e. has only finitely many isomorphism classes of indecomposable objects up to shift;

2) There is a stable component with finite Dynkin tree class;

3) There is a bounded stable component.

In this case, the Auslander-Reiten quiver consists either only of one component $ \Z [T]$ with $T \not = A_1$ a finite Dynkin diagram, or of infinitely many components $ A_1$.

In the first case $A$ is derived equivalent to $kT$,
which is a hereditary algebra of finite representation type. In the second case $A$ is simple.
Finally we introduce shift periodic complexes in analogy to periodic modules. Possible tree classes for derived categories with shift periodic complexes are deduced. We show that their Auslander-Reiten quiver is either $\Z[T]$ with $T$ a finite Dynkin diagram or the component containing the shift periodic complex has tree class $A_{\infty}$.

\bigskip

I would like to thank my supervisor Dr Karin Erdmann for reading this paper. I am also grateful to Prof. Dieter Happel and Prof. Bernhard Keller for useful discussions and comments. Special thanks are due to the referee who read this paper very carefully and made many useful suggestions. 

\section{Preliminaries and Notation}
Let $A$ denote a finite-dimensional indecomposable algebra
over a field $k$ and $A$-mod the category of finite-dimensional
left $A$-modules. We denote by $\mathcal{P}$ the full subcategory
in $A$-mod of projective modules and $\mathcal{I}$ the full
subcategory of injective $A$-modules. Let $C \in \{A \mbox{-mod},\
\mathcal{P},\ \mathcal{I}\}$.

Then $Comp^{*,?}(C)$ are the complexes that are bounded above if
$*=-$, bounded below if $*=+$ and bounded if $*=b$. The homology
is bounded if $?=b$. We denote by $D^b(A)$ the bounded derived
category and by $K^{*,?}(C)$ the homotopy category.

Contractible complexes and homotopic to zero maps are related as stated in the following well-known result.
\begin{lemma}\label{contractible}Let $f:X\to Y$ be a map of complexes.

(1) The map $f$ is homotopic to zero if and only if $f$ factors
through a contractible complex.

(2) The projection $p: \cone(f) \to X[1]$ is a retraction if and
only if $f$ is homotopic to zero.
\end{lemma}
The homotopy category and the derived categories are triangulated
categories by \cite[10.2.4, 10.4.3]{We} where the shift functor
$[1]$ is the automorphism. The distinguished triangles are given
up to isomorphism of triangles by
\[ \xymatrix{X\ar[rr]^f&&Y\ar[rr]^{0\oplus {\rm
id}_Y}&&\cone(f)\ar[rr]^{{\rm id}_{X[1]}\oplus 0}&&X[1]}\] for
any morphism $f$.

\medskip

It is difficult to calculate the morphisms in the derived
category of $A$-modules. The following theorem provides an easier way to represent them.

\begin{theo}\cite[10.4.8]{We}
We
have the following equivalences of triangulated categories
\[ \begin{array}{c}
   K^{-,b}(\mathcal{P})\cong K^{+,b}(\mathcal{I})\cong D^b(A).
\end{array}
\]
\end{theo}
We identify an $A$-module $X$ with the complex that has entry $X$ in degree $0$ and
entry $0$ in all other degrees. By abuse of notation we call this
complex $X$. A complex with non-zero entry in only one degree is also called a stalk complex.
Note that $A$-mod is equivalent to a full subcategory of $D^b(A)$ using this embedding.

\smallskip

Let $N$ be a left $A$-module and $\cdots  \stackrel{d_P^2} \to P_1
\stackrel{d_P^1}\to P_0\to N$ its minimal projective resolution.
Let $N\to I_0 \stackrel{d_I^0} \to I_1  \stackrel{d_I^1}\to \cdots
$ be its minimal injective resolution. Then we denote throughout
this paper by $pN$ the complex with $(pN)^i:=P_{-i}$ and $d^i:=
d_P^{-i}$ for $i \le 0$ and $(pN)^i:=0$ for $i>0$. Similarly we
define $iN$ to be the complex with $(iN)^n:=I_{n}$ and $d^n:=
d_I^{n}$ for $n \ge 0$ and $(iN)^n:=0$ for $n<0$.

\begin{rem}\label{contractible1}
Note that all indecomposable contractible complexes in
$Comp^{-,b}(\mathcal{P})$ have up to shift the form \[ \cdots \to 0 \to P \stackrel{\id} \to
P \to 0 \to \cdots
\] for an indecomposable projective module $P$ of $A$. We denote
such a complex where $P$ occurs in degree 0 and $1$ by $\bar P$.
\end{rem}
Contractible complexes are projective in $Comp^{-,b}(\mathcal{P})$ as the following lemma shows.
\begin{lemma}\label{projective complexes}
A contractible complex in $Comp^{-,b}(\mathcal{P})$ is projective.
\end{lemma}
\begin{proof}
Let $P$ be an indecomposable projective module and $\bar P$ the
associated indecomposable contractible complex. By the previous
remark, all contractible complexes are direct sums of shifts of
such indecomposable complexes. Let $f:C \to D$ be a surjective map
of complexes in $Comp^{-,b}(\mathcal{P})$. This means that $f^i$
is surjective for all $i\in \Z$. Suppose there is a map $g: \bar P
\to D$ of complexes.  Then we construct a map $h: \bar P \to C$ as
follows. Since $P$ is projective, there is a map $q: P \to C^{0}$
such that $f^{0}  q=g^{0}$. We then set $h^1:= d_C^{0} q $ and
$h^{0}:=q$ and $h^i:=0$ for all other degrees. We visualize this
in the diagram

\[ \xymatrix{ \cdots \ar[r] & 0 \ar[rr] \ar[rd] \ar[dd] & & P \ar[rr]_{\id} \ar[rd]_q \ar[dd]_{g^{0}} & & P \ar[rr] \ar[rd]_{d_C^{0} q} \ar[dd]_{g^1} & & 0 \ar[dd] \ar[r] & \cdots  \\
\cdots \ar[rr] & & C^{-1} \ar[rr]  \ar[ld]_{f^{-1}} & & C^{0} \ar[rr] \ar[ld]_{f^{0}} & & C^{1}  \ar[ld]_{f^{1}} \ar[rr] & & \cdots \\
\cdots \ar[r]& D^{-1} \ar[rr]  & & D^{0} \ar[rr] & & D^1 \ar[rr] & & D^2 \ar[r] & \cdots } \]

Then $h$ is a map of complexes. Furthermore $f^1  h^1=  f^1
d_C^{0} q= d_D^{0} f^{0} q= d_D^{0} g^{0}=g^1$. This gives the
proof.
\end{proof}
Finally we define for a complex $X$, the complex $\sigma^{\le
n}(X)$ to be the complex with $\sigma^{\le n}(X)^i:=X^i$ for $i
\le n$ and $d_{\sigma^{\le n}(X)}^i:=d^i_X$ for $i <n$ and
$\sigma^{\le n}(X)^i:=0$ for $i >n$. We define $\sigma^{\ge n}(X)$
analogously.

\bigskip

Next we introduce Auslander-Reiten theory for triangulated categories. We state the existence conditions for Auslander-Reiten
triangles in the bounded derived category of a finite-dimensional algebra and
prove some properties that will be needed in the other sections.

\smallskip
For an introduction to triangulated categories we refer to \cite[I.1.1]{H1}.
Let $\mathcal{T}$ be a triangulated category with translation functor $T$.

\begin{defi}\cite[I.4.1]{H1} [Auslander-Reiten triangles] \label{triangle}
A distinguished triangle $X\stackrel{u } \to Y  \stackrel{ v} \to Z \stackrel{ w} \to TX$ is called an Auslander-Reiten
triangle if the following conditions are satisfied:

(1) The objects $X,\ Z$ are indecomposable

(2) The map $w$ is non-zero

(3) If $f:W \to Z$ is not a retraction, then there exists $f':
W\to Y$ such that $v f'=f$.
\end{defi}
We introduce the following conditions.
(3') If $f:X \to W $ is not a section, then there exists $f': Y\to
W$ such that $f'  u=f$.

(3'') If $f:W \to Z$ is not a retraction, then $w  f =0$.

By \cite[I.4.2]{H1} we have that the condition (1)+ (2)+(3) is
equivalent to the condition (1)+ (2)+ (3') and also to the
condition (1)+ (2)+ (3'').

\smallskip

The condition (2) is equivalent to

(2') The map $u$ is not a section.

(2'') The map $v$ is not a retraction.

\smallskip

We refer to $w$ as the connecting homomorphism of an Auslander-Reiten triangle.
We say that the Auslander-Reiten triangle $X \to Y \to Z \to TX$ starts in $X$, has middle term $Y$
and ends in $Z$. Note also that an Auslander-Reiten triangle is uniquely determined
up to isomorphisms of triangles by the isomorphism class of the
element it ends or starts with.  The Auslander-Reiten translation $\tau$ is defined as the functor on the set of all isomorphism classes of indecomposable objects that appear at the end of an Auslander-Reiten triangle to the set of indecomposable objects that appear at the start of an Auslander-Reiten sequence. Then $\tau$ sends the isomorphism class of $Z$ to  the isomorphism class of $X$.

The Auslander-Reiten translation of $A$-mod will be denoted by $\tau_A$ in this paper to avoid confusion.

Analogously to the classical
Auslander-Reiten theory we can define irreducible maps, minimal maps,
 left almost split maps and right almost split maps as in \cite[IV.1.1, IV.1.4]{ASS}.
Irreducible maps here have the same properties as in the case of
Artin algebras. (see \cite[IV.1.8, IV.1.10]{ASS})

\begin{lemma}\label{irred}
Let $N, M \in \mathcal{T}$ and let $f:N \to M$ be an irreducible map in $D^b(A)$.

(1) Let $N \stackrel{g } \to Q \to  E \to TN$ be the Auslander-Reiten triangle,
then there is a retraction $s: Q \to M $ such that $f=s  g$.

(2) Let $L \to B \stackrel{h } \to  M \to TL$ be an Auslander-Reiten
triangle, then there is a section $r: N \to B $ such that $f=h
 r$.
\end{lemma}
Let from now on $k$ be an algebraically closed field. If
$\mathcal{T}$ is a Krull-Schmidt category we define the
Auslander-Reiten quiver to be the labelled graph
$\Gamma(\mathcal{T})$ with vertices the isomorphism classes of
indecomposable objects. For two indecomposable objects $X, Y$
there are $d_{X,Y}$ arrows from $X$ to $Y$ where $d_{X,Y}:=
$dim$_{k} $Irr$(X,Y)$. (see \cite[I.4.7]{H1} and \cite[2.7]{K})

We call a connected component of $ \Gamma(\mathcal{T})$ an Auslander-Reiten component.

\medskip

From now on let $\mathcal{T}=D^b(A)$. By \cite[2.6]{K} we have
that $D^b(A)$ is a Krull-Schmidt category. Therefore Auslander-Reiten components of $D^b(A)$ are well-defined.
Using \cite[2.6]{K} combined with \cite[B.2]{Kr} shows that $Comp^{-,b}(\mathcal{P})$ and $Comp^{+,b}(\mathcal{I})$ are also Krull-Schmidt categories. They are full subcategories of the abelian category $Comp(A)$ which are closed under direct summands and extensions. We can therefore refer to \cite{A} for their Auslander-Reiten sequences.
\begin{lemma}\label{irred2} \cite [I.4.3, I.4.5]{H1}
Let $X \stackrel{ u} \to M \stackrel{v } \to  Z \to TX$ be an Auslander-Reiten triangle and
$M \cong M_1\oplus M_2$, where $M_1$ is indecomposable. Let $i:M_1
\to M$ be an inclusion and $p: M\to M_1$ a projection. Then
$v i: M_1\to Z$ and $p  u: X \to M$ are irreducible
maps.  Furthermore $u$ is minimal left almost split and $v$ is
minimal right almost split.
\end{lemma}
Let $\nu_A$ denote the Nakayama functor of $A$ and let \[\nu_A^{-1}:=\Hom_A( \Hom_k(-,k) , A).\]
We denote by $\nu$ the left derived functor of $\nu_A$ on $D^b(A)$ and by $ \nu^{-1} $
the right derived functor of $\nu^{-1}_A$. Then $\nu$ maps a
complex $X\in Comp^{b}(\mathcal{P})$ to $\nu(X) \in Comp^b(\mathcal{I})
$, where $\nu (X)^i:=\nu_A (X^i)$ and
 $d^i_{\nu (X)}:= \nu_A(d^i_X)$ for all $i\in \Z$.

The conditions for the existence of Auslander-Reiten triangles in a
triangulated category have been determined in
\cite[I.2.4]{VBR}. It is shown that a triangulated category admits Auslander-Reiten
triangles, that is for every indecomposable element $X$ there is
an Auslander-Reiten triangle that ends in $X$ and one that starts
in $X$, if and only if the category has a Serre functor.

A specialization of this result is given in the next theorem for
the case of $D^b(A)$.
\begin{theo}\label{existence}\cite[1.4]{H2}
(1) Let $Z \in K^{-,b}(\mathcal{P})$ be indecomposable. Then there
exists an Auslander-Reiten triangle ending in $Z$ if and only if
$Z \in K^{b}(\mathcal{P})$. The triangle is of the form $\nu
(Z)[-1] \to Y \to Z \to \nu(Z)$  for some $Y \in K^{-,b}(\mathcal{P})$.

(2) Let $X \in K^{+,b}(\mathcal{I})$ be indecomposable, then there exists an Auslander-Reiten
triangle starting in $X$ if and only if $X \in
K^{b}(\mathcal{I})$. The triangle is of the form $X\to Y \to \nu^{-1}(X)[1]  \to X[1]$  for some $Y \in K^{-,b}(\mathcal{P})$.
\end{theo}
From this result we deduce that the translation $\tau$ is given by
$\nu[-1]$ and $\tau$ is natural equivalence from
$K^b(\mathcal{P})$ to $K^b(\mathcal{I})$. So every
Auslander-Reiten triangle is isomorphic to
\[\nu(X)[-1] \to \cone(w)[-1] \to X \stackrel{ w} \to \nu(X)\] for $X \in
K^{b}(\mathcal{P})$ and some map $w: X \to \nu (X)$.

\medskip

Not all irreducible maps appear in Auslander-Reiten triangles.
\begin{lemma}
Let $f:B \to C$ be an irreducible map in $D^b(A)$ that does not
appear in an Auslander-Reiten triangle. Then $B, C \not \in
K^{b}(\mathcal{P})$ and $ B, C \not \in K^b(\mathcal{I})$.
\end{lemma}
\begin{proof}
By \ref{existence} and \ref{irred} it is clear that $B \not \in
K^b(\mathcal{I})$ and $C \not \in K^b(\mathcal{P})$. Let us assume
that $B \in K^b(\mathcal{P})$ and let $n\in \N$ be minimal such
that $B^n \not =0$. Then $f$ factorizes through $\sigma^{\ge
n-1}(C)$, where $C$ is represented as a complex in
$Comp^{-,b}(\mathcal{P})$. Let $f=h  g$ be this
factorization, then $g$ is not a section, as $f$ is not a section
and $h$ is not a retraction as $ \sigma^{\ge n-1}(C) \not \cong
C$. This is a contradiction to the fact that $f$ is irreducible.
Therefore $B  \not \in K^b(\mathcal{P})$. Analogously, we can show
that $C \not \in K^b(\mathcal{I})$.
\end{proof}
More results on these irreducible maps can be found in \cite{HKR}.

\section{Auslander-Reiten triangles}
In this section we deduce some properties of Auslander-Reiten triangles and introduce stable components.
The following lemma determines the relation between irreducible
maps, retractions and sections in $K^{-,b}(\mathcal{P})$ and
$Comp^{-,b}(\mathcal{P})$.
Note that by duality the same is true if we replace $K^{-,b}(\mathcal{P})$
by $K^{+,b}(\mathcal{I})$ and $Comp^{-,b}(\mathcal{P})$ by $Comp^{+,b}(\mathcal{I})$.
\begin{lemma}\label{retraction}
Let $ B, C \in Comp^{-,b}(\mathcal{P})$ be complexes
that are not contractible. Let $f:B \to C$ be a map of complexes.

(1) Let $C, B$ be indecomposable. The map $f$ is irreducible in $Comp^{-,b}(\mathcal{P})$ if and only
if $f$ is irreducible in $K^{-,b}(\mathcal{P})$.

(2) Let $C$ be indecomposable. The map $f$ is a retraction in $Comp^{-,b}(\mathcal{P})$ if and only
if $f$ is a retraction in $K^{-,b}(\mathcal{P})$.

(3) Let $B$ be indecomposable. The map $f$ is a section in
$Comp^{-,b}(\mathcal{P})$ if and only if $f$ is a section in
$K^{-,b}(\mathcal{P})$.
\end{lemma}
\begin{proof}  We first give a proof of (2). Let $f: B \to C$ be a retraction in
$Comp^{-,b}(\mathcal{P})$. Then $f$ is clearly a retraction in
$K^{-,b}(\mathcal{P})$. Let $f$ be a retraction in $K^{-,b}(\mathcal{P})$,
then there is a map $g:C\to B$ such that $f  g$ is homotopic
to $\id_C$. Therefore $f  g - \id_C$ factors through a
contractible complex $ P$ via $s: C \to P$ and $t: P \to C$. Then
$(t,f)  \binom{-s}{g}= \id_C$. As $C$ does not have a
contractible summand, we have that $f  g$ is an isomorphism.
The proof of (3) is analogous.

\smallskip

We prove (1). Let $f: B\to C$ be an irreducible map in $K^{-,b}(\mathcal{P})$, then by
(2) and (3) $f$ is also an irreducible map in
$Comp^{-,b}(\mathcal{P})$. Suppose now that $f$ is irreducible in
$Comp^{-,b}(\mathcal{P})$ and let $g h$ be homotopic to $f$ for
some $g:D\to C$ and $h: B\to D$. Then $g h- f$ factors through a contractible
complex $ P$ via $s: B \to  P$ and $ t:  P \to C$. So $f= (g,-t)  \binom{h}{s}$ factors through $D \oplus  P$ in $Comp^{-,b}(\mathcal{P})$.
Therefore $\binom{h}{s}$ is a section in $Comp^{-,b}(\mathcal{P})$ or $(g,-t)$ is a retraction in $Comp^{-,b}(\mathcal{P})$.
This means that $h$ is a section in $K^{-,b}(\mathcal{P})$ or $g$ is a retraction in $K^{-,b}(\mathcal{P})$. Therefore $f$ is irreducible
in $K^{-,b}(\mathcal{P})$.
\end{proof}
We can therefore choose for an irreducible map in $K^{-,b}(\mathcal{P})$ an irreducible map in $Comp^{-,b}(\mathcal{P})$ that represents this map. For the rest of this paper all irreducible maps in $K^{-,b}(\mathcal{P})$ or $K^{+,b}(\mathcal{I})$ will be represented by irreducible maps in $Comp^{-,b}(\mathcal{P})$ or $Comp^{+,b}(\mathcal{I})$ respectively.

\begin{lemma}\label{correlation1}
Let $$w:=0 \to E \to C \stackrel{\sigma} \to P \to 0 $$ be an Auslander-Reiten sequence in $Comp^{-,b}(\mathcal{P})$. Then $P \in Comp^b(\mathcal{P})$. We identify $\nu(P)$ with an
indecomposable complex in $Comp^{-,b}(\mathcal{P})$. Then $w$ is isomorphic to \[ 0 \to \nu (P)[-1] \to \cone(w)[-1]  \to P \to 0 .\]
\end{lemma}
\begin{proof}
We consider $w \in Hom_{D^b(A)}(P, E[1])$ and the distinguished
triangle $E \to C \to P \stackrel{w} \to E[1]$ corresponding to the
Auslander-Reiten sequence $w$.

As this sequence does not split, the map $w$ is not homotopic to
zero by \ref{contractible}. Let $M \in Comp^{-,b}(\mathcal{P})$
and let $f: M \to P$ be a map in $Comp^{-,b}(\mathcal{P})$
representing a map in $K^{-,b}(\mathcal{P})$ that is not a
retraction. Then $f$ is not a retraction in
$Comp^{-,b}(\mathcal{P})$ by \ref{retraction} and factors through
$\sigma$ in $Comp^{-,b}(\mathcal{P})$. By \ref{triangle} \[E \to C
\to P \stackrel{w} \to E[1]\] is an Auslander-Reiten triangle. It
follows from \ref{existence} that $P \in K^b(\mathcal{P})$. As $P$
is indecomposable, we have $P \in Comp^b(\mathcal{P})$ and $E
\cong \nu(P)[-1]$.
\end{proof}
The next theorem determines the relation between Auslander-Reiten
triangles in $D^b(A)$ and Auslander-Reiten sequences in
$Comp^{-,b}(\mathcal{P})$. The analogous statement for self-injective algebras was given by
\cite[2.3, 2.2]{W}.
\begin{theo}\label{correlation}
Let $P\in Comp^{b}(\mathcal{P}) $ be an indecomposable complex
that is not contractible. We identify $\nu(P)$ with an
indecomposable complex in $Comp^{-,b}(\mathcal{P})$. Let $w:P
\to \nu (P)$ be a map in
$Comp^{-,b}(\mathcal{P})$. Then
\[ 0\to \nu (P)[-1] \to \cone(w)[-1]  \to P \to 0 \] is an Auslander-Reiten
sequence in $Comp^{-,b}(\mathcal{P}) $ if and only if $w$ induces
an Auslander-Reiten triangle in $D^b(A)$.
\end{theo}
\begin{proof}
Let $w:P \to \nu (P)$ induce an Auslander-Reiten triangle in $D^b(A)$. We have an exact sequence
\[0 \to \nu (P)[-1] \to \cone(w)[-1] \stackrel{\sigma}\to P \to 0 .\] Furthermore $\nu (P)[-1]$ and $P$ are indecomposable.
Let $M $ be a complex in $Comp^{-,b}(\mathcal{P})$ and let $f:M
\to P$ be a non-split map in $Comp^{-,b}(\mathcal{P})$. Then by
\ref{retraction}, $f$ is not a retraction in
$K^{-,b}(\mathcal{P})$. Therefore there is a map $f_1:M \to
\cone(w)[-1]$ such that $\sigma f_1=f$ in $K^{-,b}(\mathcal{P}).$
Then $f-\sigma f_1$ factors through a contractible complex $P_2$.
Let $f-\sigma f_1= g  h$  where $h: M \to P_2$ and $g:P_2 \to
P$. As $P_2$ is projective in $Comp^{-,b}(\mathcal{P})$ by
\ref{projective complexes} there is a map $s:P_2 \to \cone(w)[-1]$ such that
$g=\sigma  s$. We set $f'=f_1 +s h$, then $\sigma f'=\sigma
f_1 +\sigma s h=f$. Therefore $\sigma$ is right almost split and the exact sequence is therefore an Auslander-Reiten
sequence.
Assume now that the exact sequence

\[0 \to \nu(P)[-1]  \to \cone(w)[-1]  \stackrel{\sigma}\to P \to 0 \] is an Auslander-Reiten sequence. Then the statement follows from the proof of \ref{correlation1}.
\end{proof}
Note that if $P\in Comp^b(\mathcal{P})$ is contractible, then
there is no Auslander-Reiten sequence in $Comp^{-,b}(\mathcal{P})$
that ends in $P$, as contractible complexes are  projective
objects in $Comp^{-,b}(\mathcal{P})$ by \ref{projective
complexes}.

\smallskip

We call an Auslander-Reiten component $\Lambda$ stable, if all
vertices appear at the start and at the end of an Auslander-Reiten
triangle. By \ref{existence} this is equivalent to the fact that
all vertices in the component are in $K^b(\mathcal{I})$ and
$K^b(\mathcal{P})$.

The Auslander-Reiten quiver of $D^b(A)$ does not contain
loops by \cite[2.2.1]{XZ}. So we can apply Riedtmann's Structure Theorem
\cite[p.206]{Rie}.

\begin{cor}
Let $\Lambda$ be a stable Auslander-Reiten component of $D^b(A)$.
Then $\Lambda \cong \Z[T]/I$ where $T$ is a tree and $I$ is
an admissible subgroup of aut$(\Z[T])$.
\end{cor}
It is easy to see for which algebras stable components appear.
\begin{cor}\label{trans prop}
All complexes that appear in Auslander-Reiten triangles are contained in stable components if and
only if $A$ is Gorenstein.
\end{cor}
If $A$ has finite global dimension, then the
Auslander-Reiten quiver is a stable translation quiver. If $A$ is
self-injective and not semi-simple, then the stable Auslander-Reiten components are isomorphic to $\Z[A_{\infty}]$ by \cite[3.7]{W}. The non-stable components have been determined in \cite[5.7]{HKR}.
\section{Finite and bounded Auslander-Reiten components}
In this section, we determine the tree class of bounded and finite
stable Auslander-Reiten components. We show that finite stable
components can only appear if $A$ is simple. Bounded stable
components appear if and only if the representation type of  $D^b(A)$
is finite. This is also equivalent to the fact that the
Auslander-Reiten quiver has a component with tree class finite
Dynkin. We describe the Auslander-Reiten quiver concretely in
these cases.

We start with the following easy lemma.
\begin{lemma}\label{semi-simple}
The following are equivalent:

(1) There is a stable Auslander-Reiten component of $D^b(A)$ isomorphic
to $A_1$;

(2) The bounded derived category of $A$ has an Auslander-Reiten
triangle with middle term zero;

(3) The algebra $A$ is simple;

(4) The Auslander-Reiten quiver of $D^b(A)$ is the union of
infinitely many components $A_1$.
\end{lemma}
\begin{proof}
Obviously, part (2) follows from (1). If $\nu(X)[-1] \to 0 \to X
\stackrel{w} \to \nu(X)$ is an Auslander-Reiten triangle, then $w$ is an
isomorphism. Also for all indecomposable objects $M \in D^b(A)$
with $M \not =X$ we have $\Hom_{D^b(A)}(M,X)=0$ and $\rad \End(X)=
0$ by \ref{triangle} (3). Let $X$ be represented by an
indecomposable complex in $Comp^{-,b}(\mathcal{P}) $. Without loss
of generality we assume that $X^0 \not =0$ and $X^i=0$ for all
$i>0$. As $X$ does not have contractible summands, the map
$d_X^{-1} $ is not surjective. Suppose $X^{-1} \not=0$. Then there
is a map $h$ from the stalk complex of $X^0$ to $X$ in
$K^{-,b}(\mathcal{P})$ given by $h^0=\id_{X^0}$. This map is
non-zero in $K^{-,b}(\mathcal{P})$ as the map $d_X^{-1}$ is not
surjective. This gives a contradiction. Assume now that $X^0$ is
not simple. Then there is a non-zero map from the stalk complex of
the projective cover of $\soc X^0$ to $X$ in
$K^{-,b}(\mathcal{P})$ that is not an isomorphism. This gives also
a contradiction. Therefore $X^0$ is simple and projective.
Furthermore $\nu_A(X^0)\cong X^0$ as $w$ is an isomorphism.
Therefore $X^0$ is injective. By \cite[1.8.5]{B1} we have that $A$
is simple and $X^0$ is the only simple module in $A$ up to
isomorphism. Therefore (3) follows from (2). Clearly (3) implies
(4) and (4) implies (1).
\end{proof}
This lemma shows that if $A$ is not simple then the stable Auslander-Reiten components are the components on which $\tau$ is an automorphism.
In $Comp^{-,b}(\mathcal{P})$ we have an analogous result.

\begin{lemma}\label{semi-simple1}
The following are equivalent:

(1) The category $Comp^{-,b}(\mathcal{P})$ has an Auslander-Reiten
sequence with contractible middle term;

(2) The algebra $A$ is simple;

(3) The Auslander-Reiten quiver of $Comp^{-,b} (\mathcal{P})$ is isomorphic to $A^{\infty}_{\infty}$.
\end{lemma}
\begin{proof}

$(2) \Rightarrow (3), (1)$
Let $A$ be simple and let $S$ be the simple $A$-module. Then $S$ is projective and injective. Then by \ref{semi-simple} the connecting homomorphism $w$ can be taken to be the identity.  Let $\bar S$ be the associated contractible complex. The Auslander-Reiten sequence ending in $S$ is then of the form  $0 \to  S[-1] \to \bar S \to  S  \to 0$, where the map $\bar S \to S$ is given by

\[ \xymatrix{ \cdots \ar[r] & 0 \ar[d] \ar[r] &S \ar[r]_{\id} \ar[d]_{\id} & S \ar[r] \ar[d]  & 0 \ar[r] \ar[d] &\cdots  \\
 \cdots \ar[r] & 0 \ar[r] &S \ar[r] & 0 \ar[r] & 0 \ar[r]  & \cdots}\] and the map from $S[-1] \to \bar S$ by

\[ \xymatrix{ \cdots \ar[r] & 0 \ar[d] \ar[r] &0 \ar[r] \ar[d] & S \ar[r] \ar[d]_{\id}  & 0 \ar[r] \ar[d] & \cdots  \\
 \cdots \ar[r] & 0 \ar[r] &S \ar[r]_{\id} & S \ar[r] & 0 \ar[r] & \cdots}\] The Auslander-Reiten quiver of $Comp^{-,b}(\mathcal{P}) $ is given by \[\cdots S[-2] \to \bar S[-1] \to  S[-1] \to \bar S \to  S  \to \bar S[1] \to \cdots \] and is isomorphic to $A^{\infty}_{\infty}$. Clearly (3) implies (1).

$(1) \Rightarrow (2) $ Suppose that  $Comp^{-,b}(\mathcal{P})$ has
an Auslander-Reiten sequence with contractible middle term. Then
by \ref{correlation1} and \ref{correlation} there is an
Auslander-Reiten triangle in $K^{-,b}(\mathcal{P})$ with trivial
middle term. By the previous lemma, $A$ has to be simple.
\end{proof}
\begin{lemma}\label{finite}
Let $ \tau(C) \stackrel{f}\to B \stackrel{g}\to C\stackrel{w}\to
\tau(C)[1]$ be a distinguished triangle in $D^b(A)$ and $ M $
an indecomposable element in $D^b(A)$. Then \[ \Hom(M,\tau(C))
 \stackrel{f^*} \to \Hom(M,B) \stackrel{g^*}\to \Hom(M,C)\] and \[ \Hom(C,M)
\stackrel{\bar g}\to \Hom(B,M) \stackrel{\bar f}\to
\Hom(\tau(C),M)\] are exact. If the triangle is an Auslander-Reiten triangle then the following holds.

(1) The map $f^*$ is injective if and only if $M[1] \not \cong  C$;

(2) The map $\bar g$ is injective if and only if $ M[-1] \not
\cong \tau(C)$;

(3) The map $g^*$ is surjective if and only if $M \not \cong  C$;

(4) The map $\bar  f$ is surjective if and only if $ M \not \cong
\tau(C)$.
\end{lemma}
\begin{proof}
By \cite[I.1.2(b)]{H1} the sequence $\Hom(M,\tau(C)) \stackrel{f^* } \to \Hom(M,B) \stackrel{g^* } \to
\Hom(M,C)$ is exact.
We assume from this point on that the considered triangle is an
Auslander-Reiten triangle. Suppose there is an $h:M \to \tau(C)$
such that $f  h=0$. Then by the axiom (TR3) of  \cite[I.1.1]{H1}
there is a map $j:M[1] \to C$ such that the following diagram of
distinguished triangles commutes:

\[ \xymatrix{ M \ar[r] \ar[d]_h & 0 \ar[r]\ar[d]& M[1] \ar[r]_{\id} \ar[d]_j &M[1]\ar[d]_{h[1]}\\
\tau(C)\ar[r]_f&B \ar[r]_g & C \ar[r]_w &\tau(C)[1].}\] If $M[1]
\not \cong C$ then $w  j=0$ by \ref{triangle} (3''). This forces
$h[1]=0$ and therefore $h=0$. So in this case $f^*$ is injective.
If $M[1] \cong C$, then $f^*$ is not injective as $w[-1]:M \to
\tau(C)$ is mapped to zero. This proves (1). If $M \cong C$ then
\ref{triangle} (2'') implies that $g^*$ is not surjective.
Surjectivity for $M \not \cong C$ uses \ref{triangle} (3). This
proves (3).

The remaining cases are proven similarly.
\end{proof}
The following version of Lemma \cite[4.13.4]{B1} holds for our setup.

\begin{lemma}\label{chain1}
Let $C$ and $B$ be indecomposable complexes such that there is a
non-zero map from $C$ to $B$ in $D^b(A)$ that is not an
isomorphism. Suppose there is no chain of irreducible maps in
$D^{b}(A)$ from $C$ to $B$ of length less than $n$.

(1) Assume $B$
lies in a component $\Lambda$ of the Auslander-Reiten quiver of
$D^b(A)$ such that all elements of $\Lambda$ are in
$K^b(\mathcal{P})$. Then there exists a chain of irreducible maps
in $\Lambda$ \[P^0 \stackrel{g^1}\to P^1 \stackrel{g^2}\to \cdots
\to P^{n-1} \stackrel{g^n}\to B\] and a map $h:C \to P^0$ with
$g^n\cdots g^1 h \not = 0$ in $K^{-,b}(\mathcal{P})$.

\smallskip

(2) Assume $C$ lies
in a component $\Lambda$ of the Auslander-Reiten quiver of
$D^{b}(A)$ such that all elements of $\Lambda$ are in
$K^b(\mathcal{I})$. Then there exists a chain of irreducible maps
in $\Lambda$ \[ C \stackrel{g^1}\to P^1 \stackrel{g^2}\to \cdots
\to P^{n-1} \stackrel{g^n}\to P^n\] and a map $h:P^n \to B$ with
$hg^n\cdots g^1 \not = 0$ in $K^{+,b}(\mathcal{I})$.
\end{lemma}

\begin{proof}
We give a proof of (1) as the proof of (2) is analogous. The proof
follows by induction on $n$. Assume first $n=1$. Then there is no
irreducible map from $C$ to $B$. Let $ \nu(B)[-1] \stackrel{ f} \to E \stackrel{ l} \to B
\to \nu(B)$ be the Auslander-Reiten triangle ending in $B$. By the
assumption, there is a non zero map $t:C \to B$, that is not an
isomorphism. Therefore we have some $\sigma:C \to E$ such that $t=
l  \sigma$. Let $E:=\bigoplus_{i=1}^m E_i$ with $E_i$ indecomposable
for all $1\le i \le m$. Then $t= \sum_{i=1}^m l_i  \sigma_i$ for
some maps $l_i:E_i \to B$ and $\sigma_i:C \to E_i$. Clearly there
exists an $1 \le s \le m$ such that $l_s  \sigma_s \not =0$ in
$K^b(\mathcal{P})$ as $t\not =0$ in $K^{b}(\mathcal{P})$. By
\ref{irred2} $g_1:= l_s$ is irreducible and $E_s$ lies in
$\Lambda$. Furthermore by the induction assumption $ \sigma_s$ is
not an isomorphism as there is no irreducible map from $C$ to $B$.
We can therefore use the same argument in the induction step on
the map $\sigma_s$.
\end{proof}
We define certain length functions of complexes. Similar length functions have been defined for example in \cite{W}. We will use the following notation.
\begin{defi}
We denote by $l_c$ the function that maps an element $B \in
Comp^b(A)$ to the length of a composition series of $ \bigoplus_{i\in
\Z}  B^i$. For $Y \in Comp^b(\mathcal{P}) $ we denote by $l(Y)$
the number of projective indecomposable summands in $\bigoplus_{i\in
\Z}  Y^i$. For $X \in K^b(\mathcal{P})$ we denote by
\[ l_p(X):=\mbox{min} \{ l(Y) | Y \in Comp^b(\mathcal{P})\mbox{ and }Y
\cong X \mbox{ in }K^b(\mathcal{P}) \}.\] We define analogously
$l_i$ for elements of $K^b(\mathcal{I})$ counting the number of
indecomposable injective modules.

\smallskip

We call a stable Auslander-Reiten component $\Lambda$ bounded if $l_p$
and $l_i$ take bounded values on $\Lambda$.
\end{defi}A complex $Y\in Comp^b(\mathcal{P})$ with $l_p(X)=l(Y)$ and $Y \cong X$ in $K^b(\mathcal{P})$ is called homotopically minimal in \cite[B2]{Kr}.  Krause also shows that such a homotopically minimal element to a complex is uniquely determined up to isomorphism of complexes.

Note that $l_p$, $l_c$ and $l_i$ are defined for complexes of a
stable Auslander-Reiten component of $D^b(A)$. Also a bounded
Auslander-Reiten component is always stable, as all elements in
the component are both in $K^b(\mathcal{P})$ and
$K^b(\mathcal{I})$.

\begin{rem}\label{subadditive}
Let $\Lambda$ be a stable component. Then by \ref{correlation} every Auslander-Reiten triangle in $\Lambda$ is isomorphic to
\[\tau (P) \to \cone(w)[-1] \to P \stackrel{w} \to \tau (P)[1]\] where
\[0 \to \tau (P) \to \cone(w) [-1] \to P \to 0 \] is a short exact sequence in $Comp^b(\mathcal{P})$ and $w: P \to \tau(P) [1]$ is a representative in $Comp^b(\mathcal{P})$ of the connecting morphism. Then $\cone(w)[-1] \in Comp^b(\mathcal{P})$.
Therefore $l_p$ satisfies $l_p(\cone(w)[-1]) \le l_p(\tau (P)) + l_p(P)$ with equality if and only $\cone(w)[-1]$ does not have a contractible summand in $Comp^b(\mathcal{P})$.
\end{rem}

\smallskip

With exactly the same proof as in \cite[4.14.1]{B1} we have
\begin{lemma}
Let $P_0, \ldots , P_{2^n-1} \in Comp^b(A)$ be indecomposable and assume that
$l_c(P_i) \le n$ for all $i$. If the maps $f_i:P_{i-1} \to P_i$ are not isomorphisms for $ 1 \le i \le 2^n-1$, then $f_{2^n-1}\cdots f_2
f_1=0.$
\end{lemma}
Suppose $f:C \to B$ is a map in $Comp^{-,b}(\mathcal{P})$ that
represents an irreducible map in $K^{-,b}(\mathcal{P})$, where $C$ and
$B$ are indecomposable complexes. Then by
\ref{retraction} $f$ is irreducible in $Comp^{-,b}(\mathcal{P})$ and therefore $f$ is not an
isomorphism.

We therefore have the following result.

\begin{cor}\label{chain2}
Let $P_0, \ldots , P_{2^n-1} \in Comp^b(\mathcal{P})$ be
indecomposable such that $l_c(P_i) \le n$ for all $i$. If the $f_i:P_{i-1}
\to P_i$ are irreducible maps in $K^{-,b}(\mathcal{P})$ for $ 1 \le i \le 2^n-1$, then $f_{2^n-1}\cdots f_2 f_1=0$.
\end{cor} We can now determine
some properties of bounded components.
\begin{theo}[bounded components] \label{bounded}
Let $\Lambda$ be a stable bounded Auslander-Reiten component of
$D^b(A)$. We assume that $A$ is not simple.  Then $\Lambda$ is the
only component of the Auslander-Reiten quiver of $D^b(A)$.
Furthermore $A$ has finite global dimension and is of finite representation
type.
\end{theo}
\begin{proof} By assumption, there is an $n \in \N$ such that $l_p(M) \le n$ and $l_i(M) \le n$ for all complexes $M \in \Lambda$.
Let $R, S \in K^{-,b}(\mathcal{P})$ be indecomposable complexes
such that there are non zero maps $g: R \to N$ and $f:M \to S$  in
$K^{-,b}(\mathcal{P})$ for some $N, \ M \in \Lambda$. Let $u:= n 
\dim A$. Suppose there is no chain of irreducible maps from $R$ to
$N$. Then by \ref{chain1} part (1) there exists a chain of
irreducible maps of length $2^u$ in $\Lambda$ that is not zero.
This is a contradiction to \ref{chain2} as $l_c$ takes values at
most $u$ on $\Lambda$. Therefore there is a chain of irreducible
maps from $R$ to $N$. So $R \in \Lambda$. Analogously $S$ lies in
the component $\Lambda$. For any complex $M$ the connecting homomorphism $w$ is a non
zero map from $M$ to $\tau(M)[1]$.
We also have that $\tau(M) \in \Lambda$ as
the middle term of an Auslander-Reiten triangle is not trivial by \ref{semi-simple}.
Therefore $\tau(M), \tau(M)[1] \in \Lambda$. Thus the $[1]$ shift
acts on the component.

Let $A= \bigoplus_{i=1}^n P_i$ be a decomposition of $A$ into
indecomposable projective summands $P_i$. Let $C$ be an element of
$\Lambda$.
As $[1]$ acts on $\Lambda$ we can assume without loss of
generality that there exists a non zero map $f$ in $D^b(A)$ from
$P_i $ to $C$ for some $ 1\le i \le n$.
Therefore $P_i \in \Lambda$ and as $A$ is indecomposable we have
$P_j \in \Lambda$ for all $ 1\le j \le n$.
For all indecomposable elements $X$ in $K^{-,b}(\mathcal{P})$ there is
an $s_x \in \Z$ such that there is a non-zero map $P_i \to
X[s_x]$. Therefore $X[s_x] \in \Lambda$ using the first part of
the proof. As
the $[1]$ shift acts on the component, every indecomposable
complex of $D^b(A)$ belongs to
$\Lambda$. Therefore $\Lambda$ is the Auslander-Reiten quiver.
The stalk complex of an indecomposable $A$-module $U$ is in $\Lambda$.
A stalk complex in $D^b(A)$
can be identified with a complex in $K^{b}(\mathcal{P})$ if and only if
the stalk has a finite projective resolution. Therefore $A$ has finite global dimension as all indecomposable
complexes in $K^{-,b}(\mathcal{P})$ are in $K^b(\mathcal{P})$. As
the dimension of the indecomposable $A$-modules are bounded, we
know by \cite[1.5]{ARS} that $A$ has finite representation type.
\end{proof}
We can now determine finite components.
\begin{theo}[finite components]\label{finite comp}
Let $\Lambda$ be a finite Auslander-Reiten component of $D^b(A)$
such that all elements in $\Lambda$ belong to $K^b(\mathcal{P})$.
Then $A$ is simple and $\Lambda$ is isomorphic to $A_1$.
\end{theo}
\begin{proof}Suppose that $A$ is not simple.
As $\Lambda$ is a finite component and all vertices of $\Lambda$
are in $K^b(\mathcal{P})$, the translation $\tau$ is an
automorphism on $\Lambda$. Therefore the component $\Lambda$ is
stable and bounded.  By \ref{bounded} the shift $[1]$ acts on
$\Lambda$ which is a contradiction, as $\Lambda$ contains only
finitely many vertices. Therefore $A$ is simple and $\Lambda$
is isomorphic to $A_1$ by \ref{semi-simple}.
\end{proof}
If $A$ has finite global dimension then $D^b(A)\cong
K^b(\mathcal{P})$, and \ref{finite comp} gives the next corollary.
\begin{cor}Let $A$ be a finite-dimensional indecomposable algebra of finite global dimension.
 Suppose that the Auslander-Reiten quiver of $D^b(A)$ has a finite component $\Lambda$, then $A$ is simple.
\end{cor}

In the case of $Comp^{-,b}(\mathcal{P})$ there are no finite
components.

\begin{cor}
There is no finite Auslander-Reiten component $\Lambda$ of the
Auslander-Reiten quiver of $Comp^{-,b}(\mathcal{P})$, such that
all elements in $\Lambda$ are in $Comp^b(\mathcal{P})$.
\end{cor}
\begin{proof}Suppose $\Lambda$ is a finite component such that all elements in $\Lambda$ are in $Comp^b(\mathcal{P})$. By \ref{correlation}
we have that the corresponding Auslander-Reiten component in
$D^b(A)$ is finite. Therefore $A$ is simple by \ref{finite
comp}. But by \ref{semi-simple1} we have $\Lambda \cong
A^{\infty}_{\infty}$ which is a contradiction to the finiteness of
$\Lambda$.
\end{proof}
In analogy to the theory of algebras, we say that $D^b(A)$ has
finite representation type if and only if $D^b(A)$ has finitely
many indecomposable complexes up to shift. We call an
indecomposable complex $X$ in a stable Auslander-Reiten component
shift periodic, if there are $m \in \Z$, $n\in \N$ such that
$\tau^n(X)=X[m]$. If we have shift periodic modules in a
component, we can construct subadditive functions.

Let $\Lambda$ and $C$ be two Auslander-Reiten components of $D^b(A)$. Then we denote by $\Hom_{D^b(A)}(C,\Lambda[i])$
the union of $\Hom_{D^b(A)}(M,N[i])$ for all objects $N \in \Lambda$ and $M \in C$.
\begin{theo}\label{periodic, shift}
Let $C$ be a stable component of the Auslander-Reiten quiver of
$D^b(A)$. Suppose there is a complex $X\in C$ that is shift
periodic. Let $T$ be the tree class of $C$, then $T$ is a finite
Dynkin diagram or $A_{\infty}$.

(a) If $T$ is a finite Dynkin diagram, then the Auslander-Reiten
quiver is equal to $C$ and $D^b(A)$ has finite representation
type.

(b) Suppose $Q$ is a stable component of the Auslander-Reiten
quiver of $D^b(A)$, that is not a shift of $C$. If the set
$\Hom_{D^b(A)}(C,Q[i])$ or $\Hom_{D^b(A)}(Q[i],C)$ is non zero for
some $i\in \Z$, then the tree class of $Q$ is either Euclidean or
infinite Dynkin.
\end{theo}
\begin{proof} Let $n \in \N$, $m \in \Z$ be such that $\tau^n(X)= X[m]$.
We consider the following subadditive function for all $M \in C$

$$d(M):=\sum_{i=1}^n \sum_{j \in \Z} \dim_k \Hom_{D^b(A)}(\tau^i(X), M[j]).$$

This function takes finite values as for two complexes $L, N  \in
K^b(\mathcal{P})$ the set $\Hom_{K^b(\mathcal{P})}(L,N[s])$ is
non zero only for finitely many values of $s \in \Z$.

\smallskip

Let $\tau(D) \stackrel{f} \to B \stackrel{g} \to D \to
\tau(D)[-1]$ be an Auslander-Reiten triangle. By \ref{finite} we
have $d(\tau(D))+d(D) \le d(B)$. As $C$ contains the element $X$,
the function $d$ is a subadditive function that is not additive by
\ref{finite} (3). Therefore $T$ is a finite Dynkin diagram or
$A_{\infty}$ by  \cite[p.289]{HPR}. The shift $[m]$ induces an automorphism
of finite order on $T$ as $[m]$ commutes with $\tau$. Then we have for all complexes $M \in
C$ that $M[l]=\tau^t(M)$ for some $t \in \N$ and $l \in \Z$. It follows that on each $\tau$-orbit of $C$, the length
function $l_p$ is bounded. If $T$ is finite there are
finitely many $\tau$-orbits and therefore $C$ is bounded. Then $C$ is the only component of the Auslander-Reiten
quiver by \ref{bounded}. Furthermore $D^b(A)$ has only finitely many complexes up
to shifts.

\smallskip

Without loss of generality, let $\Hom_{D^b(A)}(X,L)\not = 0$
 for some $L \in Q$. We define $d(M)$ as above for all $M \in
 Q$. Then $d$ is an additive function by \ref{finite} on $Q$ and therefore the
 tree class of $Q$ is Euclidean or infinite Dynkin.
\end{proof}
Note that if we consider the Auslander-Reiten quiver of $A$ then
by \cite[VII 2.1, VI 1.4]{ARS} the Auslander-Reiten quiver has a finite component if
and only if there is a component whose modules have bounded dimension. In this case $A$ has
finite representation type and the Auslander-Reiten quiver
consists of only one finite component. The stable Auslander-Reiten
quiver has then tree class a finite Dynkin diagram.

The analogous theorem for the Auslander-Reiten quivers of $D^b(A)$ that have bounded
components is given next.
\begin{theo}\label{Dynkin}
The following are equivalent:

(1) The Auslander-Reiten quiver of $D^b(A)$ has a bounded
component.

(2) The representation type of $D^b(A)$ is finite.

(3) The Auslander-Reiten quiver of $D^b(A)$ has a stable component whose tree class is finite
Dynkin.
\end{theo}
\begin{proof}
(1) $\Longrightarrow$ (2)

Let $C$ be a bounded component. Let $M \in C$ be represented by an indecomposable complex in $Comp^b(\mathcal{P})$.
There is an $n_0 \in \Z$ such that $M^{n_0} \not = 0$ and $M^i = 0$
for all $ i> n_0$. Take an indecomposable summand $P$ of $M^{n_0}$. If $P[-n_0] \not = M$ there is a map $\psi: P[-n_0] \to M$ induced by the embedding of $P $ into $M^{n_0}$. This map is non-zero in $D^b(A)$, as $M$ does not have contractible direct summands. By the proof of \ref{bounded} we have $P[-n_0] \in C$.

If we represent all elements in $C$ by indecomposable
complexes in $Comp^b(\mathcal{P})$ then there is an $n \in \N$ such that $l_c$ takes values
$\le n$ for all elements in $C$. By \ref{chain2} every chain of irreducible maps of length $> 2^n$
is therefore zero.

As $P[-n_0] \in C$ we know by \ref{chain1} that there is a chain of irreducible maps of length at most $2^n$
that connects $P[-n_0]$ and $M$. There are only finitely many elements in the component $C$ that are connected to $P[-n_0]$ by
a chain of irreducible maps of length $\le 2^n$.
Therefore there are only finitely many indecomposable
complexes $L$ in $D^b(A)$ such that $L^{n_0}$ contains $P$ as a
summand and $L^i=0$ for $i>n_0$.

Therefore $D^b(A)$ has finite representation type.

\smallskip

(2) $\Longrightarrow$ (3)

Let $D^b(A)$ have finite representation type. Let $N$ be an indecomposable $A$-module. Then the complexes $\sigma^{\ge n}(pN) $ are indecomposable for all $-n \in \N$. If $N$ has infinite projective dimension, then they are pairwise non-isomorphic in $K^{-,b}(\mathcal{P})$ up to shift. The same holds for a module with infinite injective dimension if we consider $\sigma^{\le n}(iN)$ for $n \in \N$. Therefore $A$ has finite global dimension and all Auslander-Reiten components are bounded
and stable. If there is a finite component then by \ref{finite
comp} the Auslander-Reiten component consists of copies of $A_1$.
Otherwise the Auslander-Reiten quiver has only one component by \ref{bounded}. This component then
contains a shift periodic module. Therefore the tree class is finite
Dynkin or $A_{\infty}$ by \ref{periodic, shift}. As $[1]$ acts as the
identity on $A_{\infty}$ such a component can only occur if the
representation type of $D^b(A)$ is not finite. Therefore the
Auslander-Reiten quiver consists of one component $\Z [T]$ where
$T$ is a finite Dynkin diagram.

\smallskip

(3) $\Longrightarrow$ (1)

Suppose now that $D^b(A)$ has a stable Auslander-Reiten component
whose tree class $T$ is a finite Dynkin diagram. We assume this component is not finite. Then there is a component
isomorphic to $\Z [T]$
with tree class $T$ a finite Dynkin diagram. Then $T$ is either $A_n$ for $n\ge 2$, $D_n$ for $n \ge 3$, $E_6$, $E_7$ or $E_8$.
We index the vertices in $\Z[T]$ by pairs $(t,i)$ where $i \in \Z$ denotes the $i$-th
copy of $T$ and $t$ denotes the vertex of $T$. By \ref{subadditive} the function $l_p$ is subadditive
on stable Auslander-Reiten components.

We assume that
$l_p$ is not additive for only finitely many Auslander-Reiten
triangles. Then we can choose an $l \in \Z$ such that $l_p$ is
additive for all Auslander-Reiten triangles with vertices $( t,j)$ where $j > l $ and $t $ any vertex
of $T$. We denote $x_{t,i}:=l_p((t,i)) $. Let the
values $x_{t,j}$ be given for a fix $j >l$ and all vertices $ t $ of $T$.

If $T=A_n$  then we can calculate the values of $x_{t, j+1}$ from
the left to the right as follows:

\[ \xymatrix{x_{1,j} \ar[r] & x_{2,j} \ar[ld] \ar[r] & \cdots \ar[r]  \ar[ld]& x_{n-1,j} \ar[r] \ar[ld]& x_{n,j} \ar[ld]
\\ x_{2,j}-x_{1,j} \ar[r] & x_{3,j}-x_{1,j} \ar[r] & \cdots \ar[r] & x_{n, j}-x_{1,j} \ar[r]  & x_{n,j+1} }\]
Clearly this gives a contradiction as $l_p$ cannot be additive on
the Auslander-Reiten triangles ending in $(n,j+1)$. Therefore $l_p$
is not additive for infinitely many Auslander-Reiten triangles.

If $T=D_n$ we
have the following

\[\xymatrix@M=0.5pt@W=0.5pt@R=0.5pt{   &  & &  & x_{n-1,j} \ar[ldddd] \\
 x_{1,j} \ar[r] & x_{2,j}\ar[lddd] \ar[r] & \cdots \ar[r] \ar[lddd]  & x_{n-2,j}\ar[lddd] \ar[ru] \ar[rd] \\
  & & & & x_{n,j} \ar[ldd] \\
  & &  & & x_{n-1,j}-x_{1,j} \\
x_{2,j}-x_{1,j} \ar[r] & x_{3,j}- x_{1,j} \ar[r] & \cdots \ar[r]  & x_{n-1,j}+x_{n, j}-x_{1,j} \ar[ru] \ar[rd] \\
  & & & & x_{n,j}-x_{1,j}. }\]
Then the values $x_{n,i}$ are strictly decreasing for strictly
increasing $i>j$. This is a contradiction as they have to be
positive integers for all $i \in \Z$.

\medskip

For $E_6$, $E_7$ and $E_8$ we consider the following diagram.

 
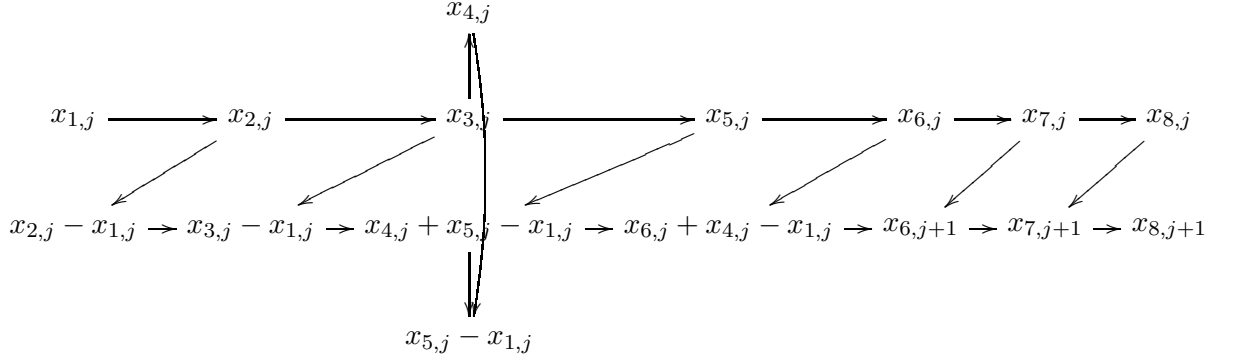
\[ \xymatrix@C=10pt{ && x_{4,j} \ar@/^/[ddd] &&&& \\ x_{1,j} \ar[r] & x_{2,j} \ar[ld] \ar[r] & x_{3,j} \ar[u] \ar[ld] \ar[r] & x_{5,j} \ar[r] \ar[dl] & x_{6,j} \ar[r] \ar[ld] & x_{7,j} \ar[r] \ar[ld] & x_{8,j} \ar[ld] \\
x_{2,j}-x_{1,j} \ar[r] & x_{3,j}-x_{1,j} \ar[r] & x_{4,j}+x_{5,j}-x_{1,j} \ar[d] \ar[r] & x_{6,j}+x_{4,j}-x_{1,j} \ar[r] & x_{6,j+1}\ar[r] & x_{7,j+1}\ar[r] & x_{8,j+1}\\ && x_{5,j} - x_{1,j} &&&& }\]

In the case of $E_6$ we have $x_{6,j+1}=x_{4,j}-x_{1,j}$. If we consider $E_7$ we have $x_{6,j+1}=x_{7,j}+x_{4,j}-x_{1,j}$ and $x_{7,j+1}= x_{4,j}-x_{1,j}$. Finally for $E_8$ we have $x_{6,j+1}=x_{7,j}+x_{4,j}-x_{1,j}$, $x_{7,j+1}= x_{8,j}+x_{4,j}-x_{1,j}$ and $ x_{8,j+1}=x_{4,j}-x_{1,j}$.

In the case $E_6$ we have by the same argument that
$x_{6,j+4}=-x_{1,j}$, in the case $E_7$ we have that
$x_{3,j+20}=-x_{3,j}+x_{4,j}$ and for $E_8$ we have $x_{1,j+15}=
-x_{1,j}$. Those are negative values as $-x_{3,j}+x_{4,j}=-x_{4,j-1}$ and we obtain a contradiction
to the assumption that $l_p$ is additive on all but finitely many
Auslander-Reiten triangles in the component.

As $l_p$ is not additive for infinitely many Auslander-Reiten
triangles in $C$, there have to be infinitely many complexes that
are homotopic to zero in the Auslander-Reiten component of
$Comp^b(\mathcal{P})$ that is associated to $C$ by
\ref{subadditive}.

As there are only finitely many indecomposable complexes homotopic
to zero in $Comp^b(\mathcal{P})$ up to shift by \ref{contractible1}, we deduce that a
shift $[m]$ induces an automorphism on $\Z[T]$ for some $m \in
\N$. Therefore $\Z[T]$ is a bounded component.
\end{proof}
Note that we only require one component to be bounded or to have tree class a finite Dynkin diagram in order to deduce that the representation type of $D^b(A)$ is finite.

We can describe the Auslander-Reiten quiver and derived category
more precisely in the case of the previous theorem. Note also that not all bounded components need to be finite as it is the case for the Auslander-Reiten quiver of an algebra.
\begin{theo}\label{inj pro}
Let one of the conditions of \ref{Dynkin} be true. Then $A$ is
either simple and the Auslander-Reiten quiver of $D^b(A)$
consists of countably many copies of $A_1$ or the Auslander-Reiten quiver
consists of one component $\Z[D] $ where $D$ is a finite Dynkin diagram and $D \not = A_1$. In the second case $A$ is derived equivalent to $k D$.
\end{theo}
\begin{proof}Suppose the bounded component is finite, then the first case holds by \ref{finite}.
If the bounded component is not finite then by \ref{bounded} the
Auslander-Reiten quiver consists of only one component which needs
to be $\Z[D]$ for $D$ a finite Dynkin diagram and $D\not =A_1$ by
\ref{Dynkin}. As $D^b(A)$ is of finite representation type, it is
discrete in the sense of \cite[1.1]{V}. By \cite[Theorem A,B]{BGS} the algebra
$A$ is derived equivalent to $k Q$ where $ Q$ is a finite Dynkin
diagram. By  \cite[I.5.5]{H1} we have $\bar Q= \bar D$. Then $D^b(kQ) \cong
D^b(k D)$ by \cite[I.5.6]{H1}, which proves the theorem.
\end{proof}
Now we know that there is only one Auslander-Reiten component if $D \not= A_1$.
Let $\Gamma_0$ be the set of isomorphism classes of indecomposable objects in $D^b(A)$. Then we call $D^b(A)$ locally finite if $ \sum_{M \in \Gamma_0}\dim_k \Hom_{D^b(A)}(M,N)$  is finite for all $N \in \Gamma_0$.
So using \cite[3.1.6]{XZ} we have that $D^b(A)$ is locally finite if and only if $A$ is derived equivalent to an
hereditary algebra of finite representation type.

\end{document}